\documentclass[a4paper,11pt]{amsart}
\usepackage{latexsym, amssymb, amsfonts,amsmath}
\usepackage{amssymb,latexsym}
\usepackage{amsfonts}
\usepackage{a4wide}
\usepackage{color}
\usepackage{comment,mathtools}%,tcolorbox}

\usepackage{amsmath,amsthm,graphicx}
\usepackage[all]{xy}
\usepackage{hyperref}
\newcommand{\C}{\mathbb{C}} 
\newcommand{\Jnd}{\mathcal{J}} 
\newcommand{\e}{\varepsilon}

\newcommand{\Pol}{{\rm Pol}}

\newcommand{\be}{\begin{equation}}
\newcommand{\ee}{\end{equation}}

\newcommand{\N}{\mathbb{N}}

\newcommand{\ot}{\otimes}

\newtheorem*{thmnn}{Theorem}
\newtheorem{thm}{Theorem}[section]
\newtheorem{defin}[thm]{Definition}
\newtheorem{lemma}[thm]{Lemma}

\newtheorem{prop}[thm]{Proposition}
\newtheorem{cor}[thm]{Corollary}

\newtheorem{remark}[thm]{Remark}
\newcommand{\AuQ}{A_u(Q)}

\numberwithin{equation}{section}

\newtheorem{innercustomthm}{Lemma}
\newenvironment{customthm}[1]
{\renewcommand\theinnercustomthm{#1}\innercustomthm}
{\endinnercustomthm}

\begin{document}

\title[Cohomology of unitary quantum groups]{Second cohomology groups of the Hopf$^*$-algebras associated to universal unitary quantum groups}

\author{Biswarup Das}
\address{Instytut Matematyczny, Uniwersytet Wroc\l awski, pl.Grunwaldzki 2/4, 50-384 Wroc\l aw, Poland}
\email{biswarup.das@math.uni.wroc.pl}

\author{Uwe Franz}
\address{Laboratoire de math\'ematiques de Besan\c{c}on,
	Universit\'e de Bourgogne Franche-Comt\'e, 16, route de Gray, 25 030
	Besan\c{c}on cedex, France}
\email{uwe.franz@univ-fcomte.fr}
\urladdr{http://lmb.univ-fcomte.fr/uwe-franz}

\author{Anna Kula}
\address{Instytut Matematyczny, Uniwersytet Wroc\l awski, pl.Grunwaldzki 2/4, 50-384 Wroc\l aw, Poland}
\email{Anna.Kula@math.uni.wroc.pl}

\author{Adam Skalski}
\address{Institute of Mathematics of the Polish Academy of Sciences,
	ul.~\'Sniadeckich 8, 00--656 Warszawa, Poland
}
\email{a.skalski@impan.pl}

%\date{started 17/01/2018, Warszawa; version dated \today}
%\date{Version from Adam, December 3rd, 2020}


\begin{abstract}
	We compute the second (and the first) cohomology groups of $^*$-algebras associated to the universal quantum unitary groups of not necesarily Kac type, extending our earlier results for the free unitary group $U_d^+$.
	The extended setup forces us to use infinite-dimensional representations to construct the cocycles. %which is the main novelty of the article.
\end{abstract}

\maketitle

%\tableofcontents

\section{Introduction}
The problem of computing Hochschild cohomology groups related to Hopf algebras associated to compact quantum groups was given a strong initial impetus in the work of Collins, H\"artel and Thom (\cite{CHT}), who considered free quantum orthogonal groups, and since then has seen a lot of progress, mainly due to Bichon (see [Bic$_{1-3}$] and references therein). The original motivation of \cite{CHT} was related to trying to understand whether the von Neumann algebras of free orthogonal groups might be isomorphic to free group factors, which was only very recently disproved by Brannan and Vergnioux (\cite{BraVer}).
Since then computing the relevant cohomology groups were seen to have many connections to other parts of cohomological algebra, but also for example to the classification of quantum L\'evy processes: for these we refer to \cite{BichonNotes} and to \cite{DFKS18}.

In this article we compute the first and the second Hochschild 
cohomology of the universal unitary quantum groups of Wang and Van Daele (\cite{wang+vandaele96}), denoted $U_Q^+$, 
for almost all choices of  a positive invertible matrix  $Q \in M_d$  (with one exception appearing in the case of the matrix $Q$ with three distinct eigenvalues), extending the methods from 
\cite{DFKS18}, where only the free case, i.e.\ $Q=I$, was treated. Therein, we 
showed among other things  that 
%$$H^1 (U_{I_d}^+)\cong M_d(\C)$$% \quad 
%H^1(O_d^+)\cong \{V\in M_d(\C): V^t=-V\}\cong \C^{\frac{d(d-1)}{2}}$$
%and 
$$H^2(U_{I_d}^+) \cong sl(d) \cong \C^{d^2-1}.$$
%\quad H^2(O_d^+) \cong o(d,\C) \cong \C^{\frac{d(d-1)}{2}},$$ 
where $sl(d)$ denotes the space of $d\times d$ complex matrices with trace 
zero. We have also studied in that paper the free orthogonal quantum groups, although for them the result had already been known due to \cite{CHT, 
BichonCompositio}. Our method was in a sense much more elementary than those applied in the papers cited above; in particular it allowed us to produce in each case a concrete minimal sets cocycles which generate the second cohomology groups, each of which was built as the cup-product of $1$-cocycles with values in certain finite-dimensional modules.

The study in this paper follows first a similar line of reasoning: given a general unitary quantum group $U_Q^+$ computing the first cohomology group is easy (see Proposition \ref{prop1coh}), and then we can define what we call the defect map, whose image (a certain subspace of $M_d(\C)$) can be shown to be isomorphic to the second cohomology group. Once we identify natural restrictions that the matrices in the subspace satisfy, it remains to construct cocycles for which the defect map would attain the desired values. Here the generality with which we work in this paper makes a significant difference in comparison to \cite{DFKS18}. We show that it is in general impossible to produce sufficiently many $2$-cocycles as cup products of finite-dimensional $1$-cocycles. However it turns out that in most cases we can use a concrete infinite-dimensional representation to solve the problem, which leads us to the main result of this paper.

\begin{thmnn}
Let $d \in \N$ and let $Q\in M_d(\C)$ be a strictly positive matrix, whose eigenvalue list is not of the form $(p, pq,pq^2)$ with $p>0, q\in(0, \infty) \setminus\{1\}$. 
Then \[ H^2(U_Q^+) \simeq sl_Q(d),\]
where $sl_Q(d) = \{A \in M_d(\C): AQ = QA, \; \textup{Tr}(AQ) = \textup{Tr}(AQ^{-1})=0\}$.	
\end{thmnn}

The detailed plan of the paper is as follows: after we conclude this introduction, in Section \ref{sec:prelim} we recall basic facts concerning on one hand the universal unitary quantum groups and the representations of the associated $^*$-algebras, and on the other abstract Hochshild cohomology of unital $^*$-algebras. In Section 3 we discuss $1$-cocycles, including those related to not-necessarily trivial bimodules. Finally Section 4 contains the main results of the paper: we first introduce the defect map, and then compute its image, leading to the final theorem.

\section{Preliminaries} \label{sec:prelim}

Throughout we will work with complex unital $^*$-algebras, and by a representation of such an algebra $A$ we will understand a unital $^*$-homomorphism from $A$ to $B(H)$, where $H$ is a Hilbert space. We write $\N_0$ for $\N \cup \{0\}$.

If $d \in \N$ and $Q \in M_d(\C)$ is a strictly positive matrix, we will call another matrix $A \in M_d(\C)$ a \emph{$Q$-matrix} if it commutes with $A$; once we find  a basis with respect to which
$Q$ is diagonal, so that $Q=\sum_{i=1}^n q_i P_{i}$, with $q_i>0$, $q_i \neq q_j$ for $i \neq j$, $i,j =1,\ldots n$, and $(P_{1}, \ldots, P_{n})$ orthogonal projections summing up to $1$, the $Q$-matrices are those which are block-diagonal with respect to the decomposition given by $(P_{1}, \ldots, P_{n})$. We will also use this terminology for matrices in $M_d(X)$ when $X$ is an arbitrary vector space. Recall that for every $d \in \N$ we denote the set of all matrices in $M_d$ of trace $0$  by $sl(d)$. By analogy we denote by  $sl_Q(d)$ the set of all $Q$-matrices in $M_d$ of the `$Q$-trace' and  `$Q^{-1}$-trace' equal $0$: 
\begin{equation}
\label{slQd} sl_Q(d) = \{A \in M_d(\C): AQ = QA, \textup{Tr}(AQ) = \textup{Tr}(AQ^{-1})=0\}.
\end{equation}
Note that for $Q$ with eigenspaces of dimensions $(d_1, \ldots, d_n)$ we have a natural inclusion
\[sl(d_1)\oplus \ldots \oplus sl(d_n) \subset 
sl_Q(d).  \]
\subsection*{Universal unitary quantum groups}

For a matrix $F\in GL_d(\C)$ ($d \in \N$)  the \emph{universal unitary quantum group} (in the sense of
Banica, cf.\ \cite{banica97}) %  denoted $U^+_F$, is defined via the identification $A_u(F)= \Pol (U^+_F)$, where  $A_u(F)$ is 
was defined via the universal unital $^*$-algebra generated by $d^2$ elements $u_{jk}$ ($j,k=1,2,...d$) such that the matrix
$u:=(u_{jk})_{j,k=1}^d$ is unitary and its conjugate, $\bar{u}$, is similar to a unitary via $F$. More specifically, the following conditions hold: 
\[\label{eq_rel1_auf}
{\mbox{(R1)}} \ uu^*=I=u^*u ; \qquad
{\mbox{(R2)}} \
F\bar{u}F^{-1}(F\bar{u}F^{-1})^*=I=(F\bar{u}F^{-1})^*F\bar{u}F^{-1}.
\]
 Let us observe that the relation (R2) is equivalent to the equality 
\[ \label{eq_rel3_auf}\tag{R3}
u^tQ\overline{u}Q^{-1}=I=Q\overline{u}Q^{-1}u^t.
\]
with the positive matrix $Q=F^*F$. Thus the algebra above is
isomorphic to the $^*$-algebra canonically associated with the universal unitary quantum group $U^+_Q$ in the sense of Wang and Van Daele, cf.\
\cite{wang+vandaele96}, i.e.\ the universal unital $^*$-algebra
$A_u(Q)$
generated by the coefficients of $u:=(u_{jk})_{j,k=1}^d$ satisfying (R1) and
(R3). It follows from Proposition 6.4.7 in \cite{timmermann} that up to an isomorphism of $^*$-algebras (of compact matrix type) without loss
of generality we can (and will) assume that the matrix $Q$ is diagonal and write for short $Q_j=Q_{jj}$, $j=1,\ldots,d$. In that case, the relation \eqref{eq_rel3_auf} reads 
\begin{equation}\label{eqn_rel_Au}
\sum_{p=1}^{d}\frac{Q_{p}}{Q_{k}} u_{p j}u^*_{pk}=\delta_{jk}1, \qquad \;\;
\sum_{p=1}^{d}\frac{Q_{j}}{Q_{p}}u^*_{jp}u_{kp}=\delta_{jk}1, 
\;\;\; j,k=1,\ldots, d.
\end{equation}
The algebra $A_u(Q)$ is naturally equipped with the standard comultiplication (which will not play a role in this paper).
We will also  write $A_u(d):=A_u(I_{d})$, $U^+_d:=U^+_{I_d}$ and call $U^+_d$ the \emph{free unitary quantum group}.
Note that we have a distinguished character (one-dimensional representation) $\epsilon: A_u(Q) \to \C$ given by the formula 
\[\epsilon(u_{jk}) = \delta_{jk}, \;\;\; j,k=1,\ldots, d.\]
Given a positive diagonal matrix $Q$ with eigenspaces of dimensions $(d_1, \ldots, d_n)$ one has natural inclusions of quantum groups $U^+_{d_1}, \ldots, U^+_{d_n}$ into $U_Q^+$, manifesting themselves as surjective unital $^*$-homomorphisms $\pi_i:A_u(Q) \to A_u(d_i)$, with $i=1,\ldots, n$, with the respective algebras viewed as the algebras of `polynomials' on respective quantum groups. The homomorphisms $\pi_i$ are given by mapping generators in the appropriate matrix block to generators of the smaller algebra, and other generators to $0$ or $1$, depending on whether they lie on the diagonal or not.

By the universality property there is a one-to-one correspondence between representations $\pi$ of $A$ on a Hilbert space $H$ and unitary matrices $R \in M_d(B(H))$ such that 
\begin{equation}R^t Q\bar{R} Q^{-1} = I = Q \bar{R} Q^{-1} R^t, \label{RQmatrix} \end{equation}
with $R_{jk} = \pi(u_{jk}), j,k=1,\ldots, d$.
In fact all finite-dimensional representations of $\AuQ$ have a particular, `block-diagonal' form: for such representations $R$ is a $Q$-matrix in the sense discussed above. This follows for example from \cite[Theorem 4.9]{Soltan}; we refer for the details to the note \cite{AABU}. Note that for a  unitary $Q$-matrix $R\in M_d(B(H))$ the condition \eqref{RQmatrix} reduces to $ R^t \bar{R}  = I = \bar{R} R^t$.
For brevity, we will further call representations $\pi$ for which the matrix $R$ is a $Q$-matrix simply \emph{$Q$-representations}.

Here we record the existence of certain infinite-dimensional representations, which will be of use later. Let $q\in (0,1)$ and let $\alpha, \gamma$ denote the standard generators of the algebra $\Pol(SU_q(2))$, as introduced in \cite{wor1}, i.e.\ elements satisfying the commutation relations:
\[  \alpha \gamma= q \gamma \alpha, \;\; \alpha \gamma^* = q \gamma^* \alpha, \;\; \gamma^* \gamma = \gamma \gamma^*, \]
\[ \alpha^* \alpha + \gamma^* \gamma = 1, \;\;\alpha \alpha^* + q^2 \gamma^* \gamma = 1,\]
so that 
\[U:=\begin{pmatrix}
\alpha & - q \gamma^* \\ \gamma &  \alpha^*	
\end{pmatrix}
\]
is a unitary matrix (the fundamental representation of $SU_q(2)$). It is then well-known, and checkable via a direct computation, that $U$ satisfies the defining relations of $\AuQ$ for $Q = \textup{diag} (1, q^2)$. In other words, $SU_q(2)$ is a quantum subgroup of the corresponding $U_Q^+$.

This means that when we consider a concrete realisation of $\Pol(SU_q(2))$ on the Hilbert space $\ell^2(\N_0)$ given by the formulas (for simplicity we just use the same symbols for the relevant operators) $$\alpha e_0=0,\ \alpha e_k = \sqrt{1-q^{2k}}e_{k-1}, k \in \N, \qquad
\gamma e_l =  q^l e_l, \;\; l \in \N_0,$$
we immediately obtain a representation of $A_u(Q)$, with $Q$ as above. Similarly %in addition any unitary $W \in B(\ell^2(\N_0))$ commuting with the concrete $\alpha, \gamma$ introduced above, 
the matrix
\begin{equation}
 R=\begin{pmatrix}
I_{\ell^2(\N_0)} & 0 & 0 \\ 0 & \alpha & -q\gamma^* \\ 0 & \gamma & \alpha^*
\end{pmatrix}
\label{infdimrep}\end{equation}
determines a representation of $A_u(Q)$ on $\ell^2(\N_0)$ for $Q = \textup{diag} (p^2, 1, q^2)$, $p >0$. Note that none of the above representations is a $Q$-representation.

\subsection*{Cohomology groups for CQG algebras} 
Let $A$ be a unital $^*$-algebra with  character $\e: A \to \C$.
We will be interested in the Hochschild cohomology of  $A$  with 
trivial coefficients (i.e.\ with the bimodule in question being $\C$ with left and right action given by $\e$), and we are 
going to use the following notations ($m\in \mathbb{N}$ and $\partial$ denotes the standard boundary operator, so that
for each  $\psi \in L(A^{\otimes(m-1)};\mathbb{C})$ and $a_1, \ldots, a_m\in A$ we have 
$(\partial \psi)(a_1\otimes\cdots \otimes a_m)=\e(a_1) \psi(a_2\otimes\cdots \otimes a_m)+
\sum_{j=0}^n (-1)^j 
\psi(a_1\otimes\cdots \otimes (a_ja_{j+1})\otimes \cdots \otimes a_{m})+
\psi(a_1\otimes\cdots \otimes a_{m-1})\e(a_m)$)
\begin{align*}
C^m(A)&=C^m(A,{_\e}\C_\e)= \{c:A^{\otimes m}\to \C, \mbox{ linear}\} \quad \mbox{($m$-cochains)};\\
Z^m(A)&=Z^m(A,{_\e}\C_\e)=\{c\in C^m(A): \partial c=0\} \quad \mbox{($m$-cocycles)}; \\
B^m(A)&=B^m(A,{_\e}\C_\e)=\{c\in Z^m(A): \exists_{\psi\in C^{m-1}(A)} \ \ 
c=\partial \psi\} \quad \mbox{($m$-coboundaries}); \\
H^m(A)& = Z^m(A)/B^m(A). 
\end{align*}
In the case when $A= \AuQ$ with the standard character we will also use the notation $H^m(U_Q^+)$ for  $H^m(\AuQ)$ and speak simply of the $m$-th cohomology group of the quantum group $U_Q^+$.
Note that if we are given two algebras $A_1$, $A_2$ as above, with respective characters $\e_1, \e_2$, and a surjective $^*$-homomorphism $\pi:A_1 \to A_2$ such that $\e_1 = \e_2 \circ\pi$, then any cocycle in $Z^m(A_2)$ gives rise to a cocycle in $Z^m(A_1)$ simply by composing with an appropriate tensor power of $\pi$. Taking into account the comments in the last subsection this means that given a positive diagonal matrix $Q$ with eigenspaces of dimensions $(d_1, \ldots, d_n)$ we have a natural map from $Z^m(U^+_{d_i})$ to $Z^m(U^+_Q)$ for each $m\in \N$, $i=1, \ldots, n$.

Naturally the explicit computation of the second cohomology group of a given algebra formally amounts to first obtaining a criterion which allows one to decide when a given $2$-cocycle is a coboundary, and then 
finding means to construct sufficiently many `non-trivial' cocycles. We will now recall two results which will assist us in that matter; they also formed the key abstract tools in \cite{DFKS18}.

Given linear maps $\psi:A\to \C$ and $c:A\otimes A \to \C$ such that  $\psi(1)=0$, $c(1\otimes 1)=0$, we define $T_{c,\psi}:A \to {\rm End}(\C \oplus \ker \e \oplus \C)$ by
$$T_{c,\psi}(a) = \begin{pmatrix} 
\e(a) & c(a \otimes -) & -\psi(a) \\
0 & a \cdot - & a-\e(a) \\
0 & 0 & \e(a)
\end{pmatrix}, \;\;\; a \in A.
$$
%Note that if $c$ is normalized, then $c(1\otimes a)=c(a\otimes 1)=0$. We are 
The following result is \cite[Lemma 5.4]{BFG17}. 

\begin{lemma} \label{lem_BFG}
	Let $c$ be a 2-cocycle on $A$, i.e.~$c\in Z^2(A,_\e\C_\e)$, and let $\psi: A \to \mathbb{C}$ be linear, with $\psi(1)=0$, $c(1\otimes 1)=0$. The mapping $T_{c,\psi}$ is a homomorphism if and only if $c\in B^2(A,\C)$ and $c=-\partial\psi$.
\end{lemma}

Although we are primarily interested in cohomology with trivial coefficients, as is well-known, 2-cocycles can be constructed by combining 1-cocycles with values in suitable bimodules via the so-called cup product.

\begin{lemma}\label{cupproduct}
	Let $\pi:A\to B(H)$ be a representation of $A$ on a Hilbert 
	space $H$, so that we can view $H$ as an $A$-bimodule with a left action given by $\pi$ and the right action given by $\e$. For $\eta_1, \eta_2\in Z^1(A,{_\pi}H_\e)$ define $c(x\otimes y) = \langle 
	\eta_1(x^*),\eta_2(y)\rangle_H$ for $x,y \in A$ and extend the resulting map linearly to $A^{\otimes 2}$. Then $c\in Z^2(A)$.
\end{lemma}

\begin{proof}
Direct computation; see also for example \cite[Proposition 3.1]{FGT15}.
\end{proof}

\section{One-cocycles and the first cohomology group for $U^+_Q$}

From now on we assume that $d \in \N$ and $Q \in M_d(\C)$ is a strictly positive diagonal matrix.
% of the form 

In this section we will consider 1-cocycles on $\AuQ$. Anticipating the use of Lemma \ref{cupproduct} we will first look at 1-cocycles with respect to arbitrary Hilbert space representations.

\begin{lemma} \label{lem_QRV}
Let $\pi:\AuQ \to B(H)$ be a representation (on a Hilbert space $H$), with the associated matrix $R\in M_d(B(H))$. Then there is a one-to-one correspondence between cocycles $\eta \in Z^1(\AuQ, _\pi H _\e)$ and matrices $V\in M_d(H)$ such that
		\begin{equation} \label{eq_R*V=}
		(R^*V)^t=\bar{R}Q^{-1}V^tQ.
		\end{equation}
The correspondence is given by
\[ \eta(u_{jk}) = V_{jk}, \;\;\; j, k =1, \ldots, d. \]		
\end{lemma}

\begin{proof}
	Consider a $\pi$-$\e$-cocycle $\eta: A_u(Q) \to 
	H$. The cocycle property implies that $\eta$ is uniquely determined by two matrices with entries in $H$, $V=[\eta(u_{jk})]_{j,k=1}^d$ and 
	$W=[\eta(u_{jk}^*)]_{j,k=1}^{d}$. A direct computation using the cocycle property and the defining relations of $\AuQ$ shows that
	\begin{equation}\label{eq_WVR_H}
	RW^t=-V, \quad R^*V=-W^t, \quad V^t=-R^tQWQ^{-1}, \quad 
	QWQ^{-1}=-Q\bar{R}Q^{-1}V^t; 
	\end{equation}
for example the first equality follows from the fact that 
\[0 =  \eta(\delta_{jk}1) = \sum_{p=1}^d \eta(u_{jp}u_{kp}^*) = 
\sum_{p=1}^d \left(\pi(u_{jp}) \eta(u_{kp}^*) + \eta(u_{jp}) \e(u_{kp})^* \right)
= \sum_{p=1}^d \left(R_{jp} W_{kp} + V_{jp} \delta_{kp} \right)\]
for all $j,k=1, \ldots, d$.
	Comparing the second and the fourth displayed equalities and using the fact that $R$ is unitary we get precisely \eqref{eq_R*V=}.
%	which should be understood as $(R^*V)_{jk} = (\bar{R}Q^{-1}V^tQ)_{kj}$. 

Hence the condition \eqref{eq_R*V=} is necessary for a matrix $V$ to define a cocycle. To  see that it is also sufficient, given such a $V$  define $W=-\bar{R}Q^{-1}V^tQ$.
Then the pair $V,W$ satisfies the relations \eqref{eq_WVR_H}. Indeed, the 	fourth equality is just the definition of $W$; for the third we multiply the 
	fourth equation by $R^t$ (from the left) getting 
	$$R^tQWQ^{-1}=-\underbrace{R^tQ\bar{R}Q^{-1}}_{=I \textup{ by } \eqref{RQmatrix}}V^t=-V^t;$$	the second equality follows from  \eqref{eq_R*V=}; composing both sides 
	of the second equality with $R$	leads to the first one. 

Therefore the mapping $\eta:A_u(Q) \to \C$ defined as
$$\eta(u_{jk})=V_{jk}, \quad \eta(u_{jk}^*)=W_{jk}, \;\;\; j,k =1, \ldots,d$$
	extends, via the linearity and the cocycle property, to the whole  $A_u(Q)$.
\end{proof}

The lemma naturally applies to the case of $\pi= \e$ and  immediately yields  the first cohomology groups for $\AuQ$.

\begin{prop} \label{prop1coh}
	 Let $Q$ be a positive matrix with eigenspaces of dimensions $(d_1, \ldots, d_n)$. Then 
	$$H^1(U_Q^+)= H^1(U_{d_1}^+) \oplus \cdots \oplus
	H^1(U_{d_n}^+) \cong M_{d_1} \oplus \cdots \oplus M_{d_n}. $$
\end{prop}

\begin{proof}
	The space $H^1(U_Q^+)$ consists simply of 1-cocycles in $Z^1(U_Q^+)$. As for $\pi=\e$ the matrix $R$ is just the identity, Lemma \ref{lem_QRV} implies  that any such cocycle 
	corresponds to a $Q$-block matrix.
\end{proof}

We finish the section with another lemma which says that the matrices corresponding to cocycles with respect to $Q$-representations always take a special form (recall that finite-dimensional representations are automatically $Q$-representations).

\begin{lemma} \label{lem_QRVf}
	Let $H$ be a Hilbert space, let $\pi:\AuQ \to B(H)$ be a $Q$-representation and let $\eta \in Z_1(\AuQ, _\pi H _\e)$. Then the matrix $V\in M_d(H)$ associated with $\eta$ is a $Q$-matrix. 
\end{lemma}

\begin{proof}
Assume as usual that $Q$ is diagonal, with eigenspaces of dimensions $(d_1, \ldots, d_n)$ and corresponding eigenspaces of dimensions $(q_1, \ldots, q_n)$.
Let $R\in M_d(B(H))$ be the matrix associated with $\pi$.

Begin by expressing the formula \eqref{eq_R*V=} in the block-diagonal form. Let $j,k=1,\ldots,n$ be indices describing blocks of our $d$ by $d$ matrices; recall that off-diagonal blocks of $R$ and $Q$ vanish. Thus on one hand
\[ ((R^*V)^t)_{jk} = [(R^*V)_{kj}]^t = [R^*_{kk} V_{kj}]^t,\]
and on the other 
\[ (\bar{R}Q^{-1}V^t Q)_{jk} = \bar{R}_{jj} q_j^{-1} (V^t)_{jk} q_k = \bar{R}_{jj} q_j^{-1} [V_{kj}]^t q_k.\] 
For simplicity we denote thus by $Z$ the $H$-valued matrix $V_{k, j}$ (a submatrix of $V$), and write $X:=R_{kk}^*$, $Y:=\overline{R_{jj}}$, with both of these being unitary matrices corresponding respectively to the $k$-th and $j$-th block of $R$. We thus obtain the following formula:
	\[ (X Z)^t=\frac{q_k}{q_j} Y Z^t.\]
The operation denoted above by $^t$ is interpreted as follows: if we view $Z$ as an operator from $\C^{d_j}$ to $\C^{d_k} \otimes H$, then $Z^t$ is the obvious operator from $\C^{d_k}$ to $\C^{d_j} \otimes H$; thus both sides are interpreted as an operator from  $\C^{d_k}$ to $\C^{d_j} \otimes H$. 

We will now deduce from the above equality and the fact that $\lambda:=\frac{q_k}{q_j}\neq 1$ whenever $k \neq j$ that for such a choice of $k$ and $j$ we must have $Z=0$.
Indeed, let us rewrite the above formula once again in the form 
	\[XZ = \lambda (YZ^t)^t,\]
or, for every $p=1,\ldots,d_k$, $r=1,\ldots, d_j$,
\begin{equation} \sum_{l=1}^{d_k} X_{pl}Z_{lr} = \lambda \sum_{m=1}^{d_j} Y_{rm} Z_{pm}.\label{index}\end{equation}
Introduce now two vectors $\xi \in \C^{d_k} \ot \C^{d_j} \ot H$ and $\eta \in \C^{d_j} \ot \C^{d_k} \ot H$, and two new unitary operators $\tilde{X}\in B(\C^{d_k} \ot \C^{d_j} \ot H)$ and $\tilde{Y} \in  B(\C^{d_j} \ot \C^{d_k} \ot H)$ by the formulas
\[\xi_{(p,r)} = Z_{pr} = \eta_{(r,p)}, \;\;\;p=1,\ldots,d_k \textup{ and } r=1,\ldots, d_j, \]
\[\tilde{X}_{(p,r), (s,t)} = X_{ps} \delta_{rt}, \;\;\;p,s=1,\ldots,d_k \textup{ and } r,t=1,\ldots, d_j,\]
\[\tilde{Y}_{(r,p), (t,s)} = Y_{rt} \delta_{ps}, \;\;\;p,s=1,\ldots,d_k \textup{ and } r,t=1,\ldots, d_j\]
(of course $\tilde{X}= X \ot I$, $\tilde{Y} = Y \ot I$, $\|\xi\|=\|\eta\|$).
Then we check the following equalities ($p=1,\ldots,d_k$  and  $r=1,\ldots,d_j$):
\[ (\tilde{X} \xi)_{p,r} = \sum_{s=1}^{d_k} \sum_{t=1}^{d_j}  \tilde{X}_{(p,r), (s,t)} \xi_{(s,t)}  = 
\sum_{s=1}^{d_k} \sum_{t=1}^{d_j}  X_{ps} \delta_{rt} Z_{st} = \sum_{s=1}^{d_k}  X_{ps}  Z_{sr},\]
\[ (\tilde{Y} \eta)_{r,p} = \sum_{s=1}^{d_k} \sum_{t=1}^{d_j}  \tilde{Y}_{(r,p), (s,t)} \eta_{(s,t)}  = 
\sum_{s=1}^{d_k} \sum_{t=1}^{d_j}  Y_{rt} \delta_{ps} Z_{ts} = \sum_{t=1}^{d_j}   Y_{rt}  Z_{tp}.\]
Thus equality \eqref{index} can be rephrased as ($p=1,\ldots,d_k$  and  $r=1,\ldots,d_j$)
\[ (\tilde{X} \xi)_{p,r} = \lambda (\tilde{Y} \eta)_{r,p}, \]
	so further
\[ \|\tilde{X} \xi\| = \lambda \|\tilde{Y} \eta\|.\]
The latter is equivalent (by the comments following the definitions of $\xi,\eta, \tilde{X}$ and $\tilde{Y}$) to the equality
\[ \|\xi\| = \lambda \|\xi\|,\]
	so that in the end $\xi =0$, and thus $Z=0$, and $V$ must have a $Q$-matrix form, as initially claimed.

\end{proof}

\section{The second cohomology group for $\AuQ$}

In this, central, section of the paper, we will compute the second cohomology groups of $U_Q^+$ for \emph{most} matrices $Q$. Once again, $Q \in M_d(\C)$ ($ d \in \N$) will be a strictly positive diagonal matrix. Note that as the relations in $\AuQ$ do not change when $Q$ is multiplied by a positive scalar, we can -- and sometimes will -- use a suitable normalisation. By the number of blocks of $Q$ we will understand the number of its distinct eigenvalues.

\subsection*{The defect map}
We will begin from generalities concerning the $2$-cocycles. An element $c\in Z^2(U_Q^+)$ will be called \emph{normalised} if $c(1 \ot 1)=0$; then also $c(1\ot a)= c(a \ot 1)= 0$ for all $a \in \AuQ$. Note that all $2$-cocycles arising as cup products via Lemma \ref{cupproduct} are normalised.
Lemma  \ref{lem_BFG} implies that the fact whether a normalised cocycle $c\in Z^2(U_Q^+)$ is a coboundary is determined by the values the cocycle takes on elements of the form $u_{kj}^+ \otimes u_{rl}^+$, with $+ \in \{\emptyset, *\}$, $k,j,r,l=1,\ldots d$. Thus we shall now analyse the properties these values must satisfy.

\begin{lemma}\label{lem_cocycle_condition}
Let $c\in Z^2(U_Q^+)$ be a normalised $2$-cocycle and define for each $j,k =1 ,\ldots, d$
\begin{align}
A_{jk} &:=\sum_{p=1}^d c(u^*_{pj} \otimes u_{pk})  \label{eq_ajk}
\\
B_{jk} &:=\sum_{p=1}^d \frac{Q_j}{Q_p}c(u^*_{jp} \otimes u_{kp}) \label{eq_bjk}.
\end{align}
Then we also have (again for $j,k =1 ,\ldots, d$)
	\begin{align}
	A_{jk}&= \sum_{p=1}^d c(u_{jp} \otimes 
	u^*_{kp}),  \label{eq_z1}
	\\
	B_{jk} &=\sum_{p=1}^d 
	\frac{Q_p}{Q_k} c(u_{pj} \otimes u^*_{pk}).  \label{eq_z2}
	\end{align}
\end{lemma}

\begin{proof}
The equalities follow from computing the expression for $\partial c$ on $\sum_{p,r=1}^d u_{jp}\otimes u_{rp}^* \otimes u_{rk}$ and $\sum_{p,r=1}^d \frac{Q_p}{Q_r} u_{pj}\otimes u_{pr}^* \otimes u_{kr}$ and using the fact that $\partial c=0$. The first one was shown in \cite[Lemma 4.3]{DFKS18}.

Fix then $j,k=1,\ldots,d$. To see that the two prescriptions for $B_{jk}$ agree, we use the relations \eqref{eqn_rel_Au} and the normalization of $c$:
	\begin{align*}
	0=&\sum_{p,r=1}^d \frac{Q_p}{Q_r} \partial c(u_{pj}\otimes u_{pr}^* \otimes u_{kr}) =
	\\
	=& \sum_{p,r=1}^d \frac{Q_p}{Q_r} \left[ \e(u_{pj}) c( u_{pr}^* \otimes u_{kr}) -
	c(u_{pj}u_{pr}^* \otimes u_{kr})+
	c(u_{pj}\otimes u_{pr}^* u_{kr})-
	c(u_{pj}\otimes u_{pr}^*) \e(u_{kr}) \right]
	\\
	=& \sum_{r=1}^d \frac{Q_j}{Q_r} c( u_{jr}^* \otimes u_{kr}) -
	\sum_{r=1}^d  c(\underbrace{\sum_{p=1}^d \frac{Q_p}{Q_r} u_{pj}u_{pr}^*}_{=\delta_{jr}1} \otimes u_{kr})+
	\sum_{p=1}^d  c(u_{pj}\otimes \underbrace{\sum_{r=1}^d \frac{Q_p}{Q_r}  u_{pr}^* u_{kr}}_{=\delta_{pk}1}) \\
	&\;\;\;\;-	\sum_{p=1}^d \frac{Q_p}{Q_r}  c(u_{pj}\otimes u_{pk}^*)
	\\
	=& \sum_{r=1}^d \frac{Q_j}{Q_r}c( u_{jr}^* \otimes u_{kr}) -
	\sum_{r=1}^d  \underbrace{c(1 \otimes u_{kr})}_{=0}+
	\sum_{p=1}^d  \underbrace{c(u_{pj}\otimes 1)}_{=0}-
	\sum_{p=1}^d \frac{Q_p}{Q_k} c(u_{pj}\otimes u_{pk}^*)
	\\
	=& \sum_{r=1}^d \frac{Q_j}{Q_r}c( u_{jr}^* \otimes u_{kr}) -
	\sum_{p=1}^d \frac{Q_p}{Q_k} c(u_{pj}\otimes u_{pk}^*).
	\end{align*}\end{proof}

The next proposition allows us to identify the coboundaries, using Lemma \ref{lem_BFG}.

%%%%%%%%%%%%%%%%%%%%%%%%%%%%%%%%%%%%%%%%%%%%%

\begin{prop} \label{lem_coboundary_condition}
Let $c\in Z^2(U_Q^+)$ be a normalised $2$-cocycle. Then the following conditions are equivalent:
\begin{itemize}
\item[(i)] $c$  is a coboundary;
\item[(ii)] for any $j, k=1,\ldots,d$ such that $Q_j= Q_k$ we have $A_{jk}=B_{kj}$, where the latter numbers were defined in the previous lemma  -- in particular $A_{kk} = B_{kk}$.
\end{itemize}
\end{prop}

\begin{proof} 
	($\Rightarrow$) Assume that $c$ is a coboundary, i.e.\ there exists a functional $\psi:A\to \C$  such that 
	$c=\partial (-\psi)$ and $\psi(1)=0$.  Fix $j,k=1,\ldots,d$ and  let us denote for short
	\begin{equation}
	\lambda_{jk}:=\psi (u_{jk}), \quad \mu_{jk}:=\psi (u_{jk}^*).
	\end{equation}
	Then
	\begin{align*}
	A_{jk} &=\sum_{p=1}^d c(u^*_{pj} \otimes u_{pk}) = \sum_{p=1}^d \partial (-\psi) (u^*_{pj} \otimes u_{pk}) 
	= \sum_{p=1}^d [ -\delta_{pj} \psi (u_{pk}) -  \psi (u^*_{pj})\delta_{pk}+  \psi (u^*_{pj} u_{pk})]
	\\ &= - \lambda_{jk} -  \mu_{kj} +  \psi (\sum_{p=1}^d u^*_{pj} u_{pk}) = -\lambda_{jk} -\mu_{kj},
\end{align*}
\begin{align*}
	B_{jk}& = \sum_{p=1}^d \frac{Q_{j}}{Q_{p}}\ c(u^*_{jp} \otimes u_{kp}) 
	= \sum_{p=1}^d \frac{Q_{j}}{Q_{p}}\ \partial (-\psi) (u_{jp}^* \otimes u_{kp}) 
	\\ & = \sum_{p=1}^d \frac{Q_{j}}{Q_{p}}\ [ -\delta_{jp} \psi (u_{kp}) -  \psi (u_{jp}^*)\delta_{kp}
	+ \psi (u_{jp}^* u_{kp})]
	\\ & = -\frac{Q_{j}}{Q_{j}}\ \psi (u_{kj}) -  
	\frac{Q_{j}}{Q_{k}}\ \psi (u_{jk}^*)
	+ \psi (\sum_{p=1}^d \frac{Q_{j}}{Q_{p}}\ u_{jp}^* u_{kp})
	= -\lambda_{kj} - \frac{Q_{j}}{Q_{k}}\ \mu_{jk}.
	\end{align*}
	This implies that 
	\begin{align}
	\label{eq_A_relation}
	A_{jk} +\lambda_{jk} +  \mu_{kj}&= 0, \\
	\label{eq_B_relation}
	B_{kj}+ \lambda_{jk} + \frac{Q_{k}}{Q_{j}}\ \mu_{kj}&= 0.
	\end{align}
	Hence
	\begin{align*}
	\mu_{kj} &= -A_{jk}-\lambda_{jk}, \\
	B_{kj}& =(\frac{Q_{k}}{Q_{j}} -1) \lambda_{jk} + \frac{Q_{k}}{Q_{j}} A_{jk}.
	\end{align*}
	From the last line we compute that 
	\begin{equation}
	(1- \frac{Q_{k}}{Q_{j}})\lambda_{jk} = \frac{Q_{k}}{Q_{j}}\ A_{jk}-B_{kj}.
	\end{equation}
	We see that if $Q_{k}=Q_{j}$, then the left hand side vanishes and hence we must have
	$B_{kj} = A_{jk}$. Thus this condition is necessary. 
	
	\medskip
($\Leftarrow$) Assume that (ii) holds. 	Motivated by the calculations from the first part of the theorem 
	we set for arbitrary $j,k=1,\ldots,d$
	\begin{align} \label{eq_functional_definition}
	\psi(u_{jk}) &: = \left\{ \begin{array}{ll} 
	-(1- \frac{Q_{k}}{Q_{j}})^{-1}\left( B_{kj} -\frac{Q_{k}}{Q_{j}}\ A_{jk}\right), & \mbox{ if }   Q_{k}\neq Q_{j},\\
	-\frac12 A_{jk}, & \mbox{ if } Q_{k}= Q_{j};
	\end{array}\right. \\
	\psi(u_{kj}^*) &:= \left\{ \begin{array}{ll} 
	(1- \frac{Q_{k}}{Q_{j}})^{-1}\left(B_{kj}-A_{jk}\right), & \mbox{ if }  Q_{k}\neq Q_{j},\\
	-\frac12 A_{jk}, & \mbox{ if } Q_{k}= Q_{j}.
	\end{array}\right. .
	\end{align}
	Note that  \eqref{eq_A_relation} and \eqref{eq_B_relation} are then satisfied.
	
It remains to show that $T=T_{c,\psi}$ as defined before Lemma \ref{lem_BFG} extends to a homomorphism. For that, we need to check that 
	elements $t_{jk}=T_{c,\psi}(u_{jk})$ satisfy the defining relations of $\AuQ$. Then Lemma \ref{lem_BFG} will allow us to conclude that 
	$-\partial \psi = c$.

	We first check that $T$ preserves the isometry relation, i.e.\ for all $j,k=1,\ldots, d$ we have $\sum_{p=1}^d t_{pj}^*t_{pk}=\delta_{jk}I$. 
	Indeed, 
	\begin{align*}
	\sum_{p=1}^d &t_{pj}^*t_{pk}=  \sum_{p=1}^{d}  \ T(u^*_{p j}) T(u_{pk})
	\\
	&= \sum_{p=1}^{d}  \begin{pmatrix} 
	\e(u^*_{p j}) & c(u^*_{p j} \otimes -) & \psi(u^*_{p j}) \\
	0 & u^*_{p j} \cdot - & u^*_{p j}-\e(u^*_{p j}) \\
	0 & 0 & \e(u^*_{p j})
	\end{pmatrix}
	\begin{pmatrix} 
	\e(u_{pk}) & c(u_{pk} \otimes -) & \psi(u_{pk}) \\
	0 & u_{pk} \cdot - & u_{pk}-\e(u_{pk}) \\
	0 & 0 & \e(u_{pk})
	\end{pmatrix}
	\\
	& = \begin{pmatrix} 
	\sum_{p=1}^{d}  \e(u^*_{p j})\e(u_{pk}) & (\star) & (\star \star) \\
	0 & \sum_{p=1}^{d}  u^*_{p j}u_{pk} \cdot - & \sum_{p=1}^{d}  \left( u^*_{p 
		j}[u_{pk}-\e(u_{pk})1] + [u^*_{p j}-\e(u^*_{p j})1]\e(u_{pk})\right)\\
	0 & 0 & \sum_{p=1}^{d}  \e(u^*_{p j})\e(u_{pk})
	\end{pmatrix}
	\\
	& = \begin{pmatrix} 
	\e(\sum_{p=1}^{d} u^*_{p j}u_{pk}) & (\star) & (\star \star) \\
	0 & \sum_{p=1}^{d}  u^*_{p j}u_{pk} \cdot - & \sum_{p=1}^{d}  \left( u^*_{p 
		j}[u_{pk}-\e(u^*_{p j})1]\e(u_{pk})\right)\\
	0 & 0 &  \e(\sum_{p=1}^{d} u^*_{p j}u_{pk})
	\end{pmatrix}
	\\
	& 
	= \begin{pmatrix} 
	\delta_{jk} & (\star) & (\star \star) \\
	0 & \delta_{jk} 1 \cdot - & 0\\
	0 & 0 & \delta_{jk}
	\end{pmatrix}.
	\end{align*}
	We are left to check that $ (\star) =0$ and that $ (\star\star) =0$. The first 
	fact holds due to the cocycle property, the normalization $c(1\otimes a)=0$, 
	and the fact that arguments are taken from $\textup{ker}\, \e$, since
	$$ (\star) =\sum_{p=1}^{d} \left( \e(u^*_{p j})c(u_{pk} \otimes -)+ c(u^*_{p j} 
	\otimes  u_{pk} \cdot -)\right)
	= \sum_{p=1}^{d} \left( c(u^*_{p j}u_{pk}\otimes -)+ c(u^*_{p j} 
	\otimes  u_{pk} ) \e(-)\right) = 0.
	$$
	The second formula is also true, because of \eqref{eq_functional_definition}:
	\begin{align*}
	(\star\star) & =  \sum_{p=1}^{d} \left(\e(u^*_{p j})\psi(u_{pk})+ c(u^*_{p j} 
	\otimes [u_{pk}-\e(u_{pk})1]) + \psi(u^*_{p j})\e(u_{pk})\right) 
	\\ & = \psi(u_{jk})+ \sum_{p=1}^{d} c(u^*_{p j} \otimes u_{pk}) + 
	\psi(u^*_{k j})
	= \lambda_{jk}+A_{jk}+\mu_{kj}\stackrel{\eqref{eq_A_relation}}{=}0.
	\end{align*}
	
	Next, we verify that $T$ preserves $u^tQ\bar{u}Q^{-1}=I$, i.e.\ 
	$\sum_{p=1}^{d}\frac{Q_{p}}{Q_{k}} t_{p j}t^*_{pk}=\delta_{jk}I$ for all $j,k=1,\ldots, d$. 
	We have
	\begin{align*}
	\sum_{p=1}^d & \frac{Q_{p}}{Q_{k}} t_{pj}t_{pk}^*=  \sum_{p=1}^{d}  \ \frac{Q_{p}}{Q_{k}} T(u_{p j}) T(u_{pk}^*)
	\\
	&= \sum_{p=1}^{d} \frac{Q_{p}}{Q_{k}} \begin{pmatrix} 
	\e(u_{p j}) & c(u_{p j} \otimes -) & \psi(u_{p j}) \\
	0 & u_{p j} \cdot - & u_{p j}-\e(u_{p j})1 \\
	0 & 0 & \e(u_{p j})
	\end{pmatrix}
	\begin{pmatrix} 
	\e(u_{pk}^*) & c(u_{pk}^* \otimes -) & \psi(u_{pk}^*) \\
	0 & u_{pk}^* \cdot - & u_{pk}^*-\e(u_{pk}^*)1 \\
	0 & 0 & \e(u_{pk}^*)
	\end{pmatrix}
	\\
	%& = \begin{pmatrix} 
	%\sum_{p=1}^{d}  \e(u^*_{p j})\e(u_{pk}) & (\star) & (\star \star) \\
	%0 & \sum_{p=1}^{d}  u^*_{p j}u_{pk} \cdot - & \sum_{p=1}^{d}  \left( u^*_{p 
	%j}[u_{pk}-\e(u_{pk})] + [u^*_{p j}-\e(u^*_{p j})]\e(u_{pk})\right)\\
	%0 & 0 & \sum_{p=1}^{d}  \e(u^*_{p j})\e(u_{pk})
	%\end{pmatrix}
	%\\
	& = \begin{pmatrix} 
	\e(\sum_{p=1}^{d} \frac{Q_{p}}{Q_{k}}u_{p j}u_{pk}^*) & (\diamond) & (\diamond\diamond) \\
	0 & \sum_{p=1}^{d} \frac{Q_{p}}{Q_{k}} u_{p j}u_{pk}^* \cdot - & \sum_{p=1}^{d} \frac{Q_{p}}{Q_{k}} \left( u_{p 
		j}u_{pk}^*-\e(u_{p j}1)\e(u_{pk}^*)\right)\\
	0 & 0 &  \e(\sum_{p=1}^{d}\frac{Q_{p}}{Q_{k}} u_{p j}u_{pk}^*)
	\end{pmatrix}
	\\
	& 
= \begin{pmatrix} 
\delta_{jk} & (\diamond) & (\diamond\diamond) \\
0 & \delta_{jk} 1 \cdot - & 0\\
0 & 0 & \delta_{jk}
\end{pmatrix}
\end{align*}
and, as above, 
\begin{align*} (\diamond) &=\sum_{p=1}^{d} \frac{Q_{p}}{Q_{k}}\left( \e(u_{p j})c(u_{pk}^* \otimes -)+ c(u_{p j} 
\otimes  u_{pk}^* \cdot -)\right)
\\&= \sum_{p=1}^{d} \left( c(\frac{Q_{p}}{Q_{k}}u_{p j}u_{pk}^*\otimes -)+ \frac{Q_{p}}{Q_{k}} c(u_{p j} 
	\otimes  u_{pk}^* ) \e(-)\right) = 0.
	\end{align*}
	and 
	\begin{align*}
	(\diamond\diamond) & =  \sum_{p=1}^{d} \frac{Q_{p}}{Q_{k}} \left(\e(u_{pj})\psi(u_{pk}^*)+ c(u_{p j} 
	\otimes [u_{pk}^*-\e(u_{pk}^*)]) + \psi(u_{p j})\e(u_{pk}^*)\right) 
	\\ & = \frac{Q_{j}}{Q_{k}}\psi(u_{jk}^*)+ \sum_{p=1}^{d} \frac{Q_{p}}{Q_{k}}c(u_{p j} \otimes u_{pk}^*) + 
	\frac{Q_{k}}{Q_{k}} \psi(u_{k j})
	= \frac{Q_{j}}{Q_{k}}\mu_{jk}+ B_{jk} + \lambda_{k j} \stackrel{\eqref{eq_B_relation}}{=}0.
	\end{align*}
	
	Using the same method we check that the other relations of $\AuQ$ are also satisfied; we leave the details to the reader. 
	%%%%%%%%%%%%%%%%%%%%%%%%%%%%%%%%%%%%%%%
\end{proof}

%{\color{blue}
%\begin{remark}\label{rem-for-gamma-W}
%Note that we actually proved the following stronger result $\ldots$
%
%We shall make use of this in Section \ref{sec-BM}.
%
%{\bf To be written}
%\end{remark}
%}

We are ready to define the defect map, following the construction in \cite[Subsection 4.1]{DFKS18}. 

\begin{defin} \label{def:defect}  Let $Q$ be a strictly positive diagonal matrix with $n$ distinct eigenvalues and the corresponding decomposition
$Q=\sum_{i=1}^n q_i P_i$.	
The defect map $\Delta: Z^2(U_Q^+) \to M_d$ is defined as follows: first for each $c \in Z^2(U_Q^+)$ consider a matrix $D(c)$ given by the formula 
\[ D(c)_{jk} = A_{jk} - B_{kj}, \;\;\;\;j,k=1,\ldots,d \]	
where the coefficients $A_{jk}$ and $B_{jk}$ are defined as in Lemma \ref{lem_cocycle_condition}
and then put 
\[ \Delta(c) = \sum_{i=1}^n P_i D(c) P_i.\]
\end{defin}

It is easy to see that $\Delta$ is a linear map. Its importance stems from the next corollary (of Proposition \ref{lem_coboundary_condition}).

\begin{cor} \label{cor:H2}
We have the following isomorphism:
\[ H^2(U_Q^+) \simeq\Delta(Z^2(U_Q^+)).\]
\end{cor}

\begin{proof}
It is well-known and easy to show (see for example \cite[Lemma 4.4]{DFKS18}) that when computing the second cohomology group one can consider only normalised cocycles. Proposition \ref{lem_coboundary_condition} shows that a normalised cocycle $c \in Z^2(U_Q^+)$ is a coboundary if and only if $\Delta(c)=0$. Then the statement becomes a consequence of the first isomorphism theorem.
\end{proof}

It thus remains to understand the image of the defect map. Before we formulate the general result, recall the formula \eqref{slQd}.

\begin{thm}\label{thm:image}
 Let $Q\in M_d(\C)$ be a strictly positive diagonal matrix with eigenspaces of dimensions $(d_1, \ldots, d_n)$. Then 
$$sl(d_1)\oplus \ldots \oplus sl(d_n) \subset \Delta(Z^2(U_Q^+))\subset 
 sl_Q(d).$$ 	
\end{thm}
\begin{proof}
We will first show that any matrix 	in $\Delta(Z^2(U_Q^+))$ must belong to $sl_Q(d)$. Let then $c\in Z^2(U_Q^+)$. The fact that $\Delta(c)$ is a $Q$-matrix follows from the definition of the defect map. It remains then to compute the relevant deformed traces:
\begin{align*}
{\rm Tr} (\Delta(c) Q) &= \sum_{k=1}^d (\Delta(c) Q)_{kk} = \sum_{k=1}^d 
Q_k \Delta(c)_{kk} \\& =\sum_{k=1}^d Q_k \left ( 
\sum_{p=1}^d c(u^*_{pk} \otimes u_{pk}) - 
\sum_{p=1}^d \frac{Q_k}{Q_p} c(u^*_{kp} \otimes u_{kp}) \right)
%\\ &= (\textup{by (Z1-Z2)) } 
\\&=\sum_{k=1}^d Q_k \left ( 
\sum_{p=1}^d c(u_{kp} \otimes u_{kp}^*) - 
\sum_{p=1}^d \frac{Q_p}{Q_k} c(u_{pk} \otimes u^*_{pk}) \right)
\\ &=
\sum_{p,k=1}^d Q_k c(u_{kp} \otimes u^*_{kp}) - 
\sum_{p,k=1}^d Q_p c(u_{pk} \otimes u^*_{pk})
=0,
\end{align*}
where in the fourth equality we use Lemma \ref{lem_cocycle_condition}. Similarly
\begin{align*}
{\rm Tr} (\Delta(c)Q^{-1}) &= {\rm Tr} (Q^{-1}\Delta(c)) = \sum_{k=1}^d
(Q^{-1}\Delta(c))_{kk} 
\\&= \sum_{k=1}^d \frac{1}{Q_k}   \left ( 
\sum_{p=1}^d c(u^*_{pk} \otimes u_{pk}) - 
\sum_{p=1}^d \frac{Q_k}{Q_p} c(u^*_{kp} \otimes u_{kp}) \right)
\\ & =\sum_{p,k=1}^d 
\frac{1}{Q_k} c(u^*_{pk} \otimes u_{pk}) - \sum_{p,k=1}^d 
\frac{1}{Q_p}c(u^*_{kp} \otimes u_{kp})=0.
\end{align*}

It remains then to prove the first inclusion. We claim that it follows from the proof of \cite[Theorem 4.6]{DFKS18}, where it was shown that $H^2(U^+_d)\simeq sl(d)$, and explicit cocycles realising that isomorphism were constructed,  and the remarks stated in Section \ref{sec:prelim} after the definition of the cohomology groups, noting that we have natural embeddings of $C^2(U^+(d_i))$ into $C^2(U^+(Q))$ for $i=1, \ldots, n$.
We leave verifying the details to the reader.
\end{proof}

Note that for $n \geq 3$  the inclusion $sl(d_1)\oplus \ldots \oplus sl(d_n) \subset 
sl_Q(d)$ is strict, as the dimension of the first space equals $d_1^2+\ldots + d_n^2-n$, and of the second $d_1^2+\ldots + d_n^2-2$.

\subsection*{`Block dimension 2' and defects arising from $1$-cocycles with respect to finite-dimensional representations}

We begin the section by noting that if $Q$ is a matrix with precisely two different eigenvalues, the last results already yield the second cohomology group of $U_Q^+$.

\begin{cor}\label{cor:coh2}
 Let $Q\in M_d(\C)$ be a strictly positive diagonal matrix with eigenspaces of dimensions $(d_1, d_2)$.
Then \[ H^2(U_Q^+) \simeq sl(d_1)\oplus sl(d_2).\]	
\end{cor}
\begin{proof}
The result follows from noting that $sl(d_1)\oplus sl(d_2) = sl_Q(d)$ simply because  the dimensions of two spaces match, as noted after the last theorem. Thus Corollary \ref{cor:H2} and Theorem \ref{thm:image} end the proof.	
\end{proof}

Let us then return to a general matrix $Q$, with the eigenspace dimensions $(d_1, \ldots, d_n)$, $n \geq 3$. As our way of constructing explicit $2$-cocycles is based on Lemma \ref{cupproduct}, one might ask whether one can use the cup product of $1$-cocycles with respect to finite-dimensional representations to obtain the `missing' cocycles, for which the defect map takes values outside of $sl(d_1)\oplus \ldots \oplus sl(d_n)$. The next proposition shows that the answer is negative: recall that finite-dimensional representations are automatically $Q$-representations.

\begin{prop}\label{prop:nogo}  Let $ d \in \N$ and let $Q\in M_d(\C)$ be a diagonal strictly positive matrix with the eigenspace dimensions $(d_1, \ldots, d_n)$.
Let $H$ be a Hilbert space, let $\pi:\AuQ \to B(H)$ be a $Q$-representation and let $\eta_1, \eta_2 \in Z_1(\AuQ, _\pi H _\e)$.	Let $c\in C^2(U^+_Q)$ be the cocycle constructed out of $\eta_1, \eta_2$ via Lemma \ref{cupproduct}. Then $\Delta(c) \in sl(d_1)\oplus \ldots \oplus sl(d_n)$.
\end{prop}
\begin{proof}
This is essentially a consequence of Lemma	\ref{lem_QRVf} and the definition of the defect map. Indeed, let $V, W\in M_d(H)$ be matrices corresponding to cocycles $\eta_1, \eta_2$ via Lemma \ref{lem_QRV}. Then for each $j,k =1,\ldots,d$ we have (see Definition \ref{def:defect}) 
\begin{align}\label{Dform}  D(c)_{jk} &= \sum_{p=1}^d \left(c(u^*_{pj} \otimes u_{pk}) - \frac{Q_k}{Q_p}c(u^*_{kp} \otimes u_{jp}) \right) = 
\sum_{p=1}^d \left(\langle V_{pj}, W_{pk}\rangle - \frac{Q_k}{Q_p}\langle V_{kp}, W_{jp}\rangle \right). 
 \end{align}
Due to Lemma \ref{lem_QRVf} both $V$ and $W$ are $Q$-matrices. This means that the above reduces to
\[ D(c)_{jk} = \sum_{p=1}^d \left(\langle V_{pj}, W_{pk}\rangle - \langle V_{kp}, W_{jp}\rangle \right) \textup{ if } Q_j =Q_k,\]
and $D(c)_{jk}=0$ if $Q_j \neq Q_k$. Thus $\Delta(c)=D(c)$, and finally if we compute the partial trace at a given block (say corresponding to the eigenvalue $q_l$: we put $J=\{k=1,\ldots, d: Q_k = q_l\}$) we get
\[ \textup{Tr}_l (\Delta(c)) := \sum_{k\in J} \Delta(c)_{kk} = 
\sum_{k\in J} \sum_{p \in J}  \left(\langle V_{pk}, W_{pk}\rangle - \langle V_{kp}, W_{kp}\rangle \right) =0.
\]
\end{proof}

\subsection*{The `non-degenerate' case in dimension 3}
The considerations in the last subsection show that the next case to study concerns a matrix $Q$ with three distinct eigenvalues. Proposition \ref{prop:nogo} tells us that to produce the `missing' $2$-cocycles via Lemma \ref{cupproduct} we need to consider infinite-dimensional representations.

Thus throughout this subsection we assume that $d=3$, and the matrix $Q$ has 3 distinct eigenvalues, so that we may consider without loss of generality $Q={\rm diag}\ (1,p^2,q^2)$ with 
$0<q<p<1$. We also fix throughout the representation $\pi:\AuQ \to B(\ell^2(\N_0))$  associated with the matrix
\begin{equation} \label{formula:Keyrep} R=\begin{pmatrix}
\alpha & 0 & -q\gamma^* \\ 0 & I & 0 \\ \gamma & 0 & \alpha^*
\end{pmatrix}, \end{equation}
where $\alpha$ and $\gamma$ denote the images of generators of the standard infinite-dimensional representation of $\Pol(SU_q(2))$, as discussed before the formula \eqref{infdimrep}. We begin by analysing the $1$-cocycles in $Z_1(\AuQ, _\pi \ell^2(\N_0) _\e)$.

\begin{lemma}\label{lem:cocinf}
A matrix $V=(v_{ij})_{i,j=1}^3 \in M_3(\ell^2(\N_0))$ yields a cocycle in $Z_1(\AuQ, _\pi \ell^2(\N_0) _\e)$ (via Lemma \ref{lem_QRV}) if and only if 
\begin{equation} \label{eq_set_of_equations_for_eta_infty_v12}
\begin{cases}
\left[(1+ \frac{q^3}{p^4})I-\frac1{p^2}\alpha - \frac{q^3}{p^2}\alpha^* \right]v_{12}=0 \\ 
\gamma^*v_{32}= (\frac1{p^2}- \alpha^*)v_{12} 
\end{cases}, 
\end{equation}
\begin{equation} \label{eq_set_of_equations_for_eta_infty_v2}
\begin{cases}
[(1+\frac{p^4}{q})I- \frac{p^2}{q^2}\alpha  -qp^2 \alpha^*]v_{23} =0\\
\gamma^*v_{21} =\frac1{p^2} (I-\frac{p^2}{q^2}\alpha)v_{23}
\end{cases} 
\end{equation}
and 
\begin{equation}\label{eq-set-suq2}
\begin{cases}
\alpha^*v_{11} -\frac1{q}\gamma v_{13} =\alpha^*v_{11}+ \gamma^*v_{31} \\
q^2 \alpha^*v_{31} -q\gamma v_{33} = -q\gamma v_{11}+  \alpha v_{31}\\
\gamma^*v_{11} + \frac1{q^2} \alpha v_{13} = \alpha^*v_{13}+ \gamma^*v_{33} \\
q^2\gamma^*v_{31} + \alpha v_{33} =-q\gamma v_{13}+  \alpha v_{33}
\end{cases}.
\end{equation}
\end{lemma}
\begin{proof}
According to Lemma \ref{lem_QRV} a matrix $V$ as above defines a 
$\pi$-$\e$-cocycle if and only if 
$(R^*V)^t=\bar{R}Q^{-1}V^tQ$.
In our case this means that
\begin{align*}
\bar{R}Q^{-1}V^tQ & = 
\begin{pmatrix}
\alpha^* & 0 & -q\gamma \\ 0 & I & 0 \\ \gamma^* & 0 & \alpha \end{pmatrix}
\begin{pmatrix}
1 & 0 & 0 \\ 0 & \frac1{p^2} & 0 \\ 0 & 0 & \frac1{q^2} \end{pmatrix}
\begin{pmatrix}
v_{11} & v_{12} & v_{13} \\ v_{21} & v_{22} & v_{23} \\ v_{31} & v_{32} & 
v_{33} 
\end{pmatrix}^t
\begin{pmatrix}
1 & 0 & 0 \\ 0 & p^2 & 0\\ 0 & 0 & q^2 \end{pmatrix}
\\ %%%%%%%%%%%%%%%%%%%%%%%%%%%%%%%%%%
&= \begin{pmatrix}
\alpha^* & 0 & -q\gamma \\ 0 & I & 0 \\ \gamma^* & 0 & \alpha \end{pmatrix}
\begin{pmatrix}
v_{11} & p^2v_{21} & q^2v_{31} \\ 
\frac1{p^2}v_{12} & v_{22} & \frac{q^2}{p^2}v_{32} \\ 
\frac1{q^2} v_{13} &\frac{p^2}{q^2} v_{23} & v_{33} \end{pmatrix}
\\ %%%%%%%%%%%%%%%%%%%%%
&= \begin{pmatrix}
\alpha^*v_{11} -\frac1{q}\gamma v_{13} & p^2\alpha^*v_{21} -\frac{p^2}{q}\gamma 
v_{23} & q^2\alpha^*v_{31} -q\gamma v_{33} \\ 
\frac1{p^2}v_{12} & v_{22} & \frac{q^2}{p^2}v_{32} \\ 
\gamma^*v_{11} + \frac{1}{q^2}\alpha v_{13} 
& p^2\gamma^*v_{21} + \frac{p^2}{q^2}\alpha v_{23}
& q^2\gamma^*v_{31} + \alpha v_{33}
\end{pmatrix}
\end{align*}
equals
\begin{align*}
(R^*V)^t &= 
\left( \begin{pmatrix}
\alpha^* & 0 & \gamma^* \\ 0 & I & 0 \\  -q\gamma & 0 &\alpha
\end{pmatrix} 
\begin{pmatrix}
v_{11} & v_{12} & v_{13} \\ v_{21} & v_{22} & v_{23} \\ v_{31} & v_{32} & 
v_{33} 
\end{pmatrix}
\right)^t
\\= &
\begin{pmatrix}
\alpha^*v_{11}+ \gamma^*v_{31} & \alpha^*v_{12}+ \gamma^*v_{32}  & 
\alpha^*v_{13}+ \gamma^*v_{33}  \\
v_{21} & v_{22} & v_{23} \\ 
-q\gamma v_{11}+  \alpha v_{31} & -q\gamma v_{12}+  \alpha v_{32} 
& -q\gamma v_{13}+  \alpha v_{33} 
\end{pmatrix}^t
\\= &
\begin{pmatrix}
\alpha^*v_{11}+ \gamma^*v_{31} & v_{21} & -q\gamma v_{11}+  \alpha v_{31}\\ 
\alpha^*v_{12}+ \gamma^*v_{32} & v_{22} & -q\gamma v_{12}+  \alpha v_{32}\\ 
\alpha^*v_{13}+ \gamma^*v_{33} & v_{23} & -q\gamma v_{13}+  \alpha v_{33}  
\end{pmatrix}.
\end{align*}
Comparing entries $11$, $13$, $31$ and $33$ of the matrices above yields the last set of equalities in the formulation of the lemma. Consider then entries $21$ and $23$. They yield
$$  \begin{cases}
(\alpha^*-\frac1{p^2})v_{12} + \gamma^*v_{32} =0 \\ %12
q\gamma v_{12}+  (\frac{q^2}{p^2}-\alpha)v_{32} =0
\end{cases}.$$

Using the fact that $(\alpha - 
\frac{q^3}{p^2})\gamma^* = q\gamma^*\alpha - \frac{q^3}{p^2}\gamma^*= 
q\gamma^*(\alpha - \frac{q^2}{p^2})$ we get
\begin{align*}
%& \begin{cases}
%(\alpha^*- \frac1{p^2})v_{12} + \gamma^*v_{32} =0 \quad /(\alpha - 
%\frac{q^3}{p^2}) (.) \\
%q\gamma v_{12}-  (\alpha - \frac{q^2}{p^2})v_{32} =0 \quad /q\gamma^* 
%(.) 
%\end{cases}
%\\ %%%%%%%%%%%%%%%%%%%%%%%%%%%%%
& \quad \begin{cases}
(\alpha - \frac{q^3}{p^2})(\alpha^*-\frac1{p^2})v_{12} + 
q^2\gamma^* \gamma v_{12}=0 \\
\gamma^*v_{32}=(\frac1{p^2}- \alpha^*)v_{12} 
\end{cases}.
%\\ %%%%%%%%%%%%%%%%%%%%%%%%%%%%%
\end{align*}
Using the commutation relations in $SU_q(2)$ we simplify the first formula
\begin{align*}
&\big[\underbrace{(\alpha \alpha^* +q^2\gamma^* 
	\gamma)}_{=1} -\frac1{p^2}\alpha - 
\frac{q^3}{p^2}\alpha^*+\frac{q^3}{p^2}\frac1{p^2} 
\big]v_{12}=0. %\\
%& \big[(1+ \frac{q^3}{p^4})I-\frac1{p^2}\alpha 
%- \frac{q^3}{p^2}\alpha^* \big]v_{12}=0.
\end{align*}
This gives the equalities in \eqref{eq_set_of_equations_for_eta_infty_v12}.

If we then consider entries $12$ and $32$, we obtain 

$$\begin{cases}
(I-p^2 \alpha^*)v_{21} +\frac{p^2}{q}\gamma v_{23}=0\\
p^2\gamma^*v_{21} + (\frac{p^2}{q^2}\alpha -I)v_{23} =0
\end{cases}.$$ 
Due to the relation
$\gamma^*(I-p^2 \alpha^*) = \gamma^*-qp^2 \alpha^*\gamma^* =(I-qp^2 
\alpha^*)\gamma^*$,
we obtain
\begin{align*}
&   \begin{cases}
p^2\frac{p^2}{q}\gamma^*\gamma v_{23} - (I-qp^2 \alpha^*)(\frac{p^2}{q^2}\alpha 
-I)v_{23} =0 \\
\gamma^*v_{21} = \frac{1}{p^2}(I-\frac{p^2}{q^2}\alpha)v_{23}
\end{cases}.
\end{align*}
Therefore
\begin{align*}
& \left[\frac{p^4}{q}(\gamma^*\gamma + \alpha^*\alpha) - 
\frac{p^2}{q^2}\alpha +I -qp^2 \alpha^*\right]v_{23} =0.
% \\
%& ((1+\frac{p^4}{q})I- \frac{p^2}{q^2}\alpha  -qp^2 \alpha^*)v_{23} =0.
\end{align*}
Another use of the commutation relation, this time  $\gamma^*\gamma + \alpha^*\alpha = I$, yields the equalities in \eqref{eq_set_of_equations_for_eta_infty_v2}.
\end{proof}

We shall focus on finding non-trivial solutions to the sets of equations appearing separately in  \eqref{eq_set_of_equations_for_eta_infty_v12} and \eqref{eq_set_of_equations_for_eta_infty_v2}. It can be quickly seen that once we write each of these using the canonical orthonormal basis, they become recurrence relations which admit formal solutions, and the key question is when these solutions actually represent vectors in $\ell^2(\N_0)$. Here Lemma \ref{lem_rysiek} will be of help.
It turns out that we have to treat separately the case of $p<q^{\frac{1}{2}}$ and  of $p>q^{\frac{1}{2}}$. 

\begin{prop} \label{cocycleCase1}
Let $0<q<p<1$, $p < q^{\frac{1}{2}}$. Then the set of equations  \eqref{eq_set_of_equations_for_eta_infty_v12} admits a solution $(v_{12},v_{32})$ with both  $v_{12},v_{32}$ non-zero vectors in $\ell^2(\N_0)$. Put 
$$V=\begin{pmatrix}
0 & v_{12} & 0 \\ 0 & 0 & 0 \\ 0 & v_{32} & 0 
\end{pmatrix}$$
and define the corresponding cocycle $\eta \in Z_1(\AuQ, _\pi \ell^2(\N_0) _\e)$ via Lemma \ref{lem_QRV}
(recall that $Q={\rm diag}\ (1,p^2,q^2)$, and the representation $\pi$ is defined via the matrix \eqref{formula:Keyrep}). Let $c\in Z_2(U_Q^+)$ be the cocycle defined as the cup product of 	$\eta$ with itself, as in Lemma \ref{cupproduct}. Then 
\[ \Delta(c) \textup{ is a invertible diagonal matrix},\]
where $\Delta$ is the defect map of Definition \ref{def:defect}. 	
\end{prop}
\begin{proof}
By Lemma \ref{lem:cocinf} to find a cocycle of the form given by the matrix $V$ described in the proposition we need to find a non-zero solution of the set of equations \eqref{eq_set_of_equations_for_eta_infty_v12}.

Suppose for the moment that we do have a solution, for which $v_{12}=\sum_{k=0}^\infty b_ke_k$, where $(b_k)_{k=0}^\infty \in \ell^2$. Then the first of the  equations in  \eqref{eq_set_of_equations_for_eta_infty_v12} would yield
\begin{align*}
0 & =
\sum_{k=0}^\infty (1+\frac{q^3}{p^4})b_ke_k - 
\sum_{k=1}^\infty \frac1{p^2} \sqrt{1-q^{2k}} b_ke_{k-1}
-\sum_{k=0}^\infty \frac{q^3}{p^2}b_k\sqrt{1-q^{2(k+1)}}e_{k+1}.
\\ & =
\left((1+\frac{q^3}{p^4})b_0 -\frac1{p^2} \sqrt{1-q^{2}} b_1\right) e_0
\\& \;\;\;\;\;\;\;+
\sum_{k=1}^\infty \left((1-\frac{q^3}{p^4})b_k - 
\frac{\sqrt{1-q^{2(k+1)}}}{p^2} b_{k+1} -\frac{q^3}{p^2}\sqrt{1-q^{2k}}b_{k-1}
\right) e_k .
\end{align*}
We thus see that the sequence $(b_k)_{k=0}^\infty$ is determined by $b_0\in \C$ and the recurrence relation
\begin{equation}\label{eq_bks1a}
b_1 = \frac{p^2(1+\frac{q^3}{p^4})}{ \sqrt{1-q^{2}} }b_0 
% b_1 =\frac{q^2(\frac{1}{q}+s^2)}{s\sqrt{1-q^{2}}}b_0
, \quad 
b_{k+1}
=\frac{1}{ \sqrt{1-q^{2(k+1)}}} \left[\left(p^2+\frac{q^3}{p^2}\right)b_k -
q^3\sqrt{1-q^{2k}}b_{k-1}\right].
\end{equation}
The last relation can be written in a more elegant way if we set 
$\kappa:=p^2+\frac{q^3}{p^2}$, 
$\mu:=q^3$ and $c_k:=(1-q^{2k})^{-\frac12}$, $k \in \N_0$. Then 
\eqref{eq_bks1a} takes the form
\begin{equation}\label{eq_bks2}
b_1 =\kappa_1 c_1 b_0, \quad 
b_{k+1} = \kappa c_{k+1}b_k -\mu\frac{c_{k+1}}{c_k}b_{k-1}, \;\; k \in \N.
\end{equation}
Let us then return to a formal argument: set $b_0=1$ and define $b_k$ for $k \in \N$ inductively by \eqref{eq_bks2}.
We are then precisely in the situation of  Lemma \ref{lem_rysiek}, with $a=p^2 $ and $b=\frac{q^3}{p^2}$ so that $a<q$, and the conclusion of Lemma \ref{lem_rysiek} applies, so that $(q^{-k} b_k)_{k=0}^\infty \in \ell^2$. Recall that we have $1/c_k \leq 1$ for all $k \in \N_0$; then the last conclusion means that we can indeed consider the following (non-zero!) vectors in $\ell^2(\N_0)$: $v_{12}:=\sum_{k=0}^\infty b_ke_k$ and $v_{32}:=\frac{1}{p^2} e_0+
\sum_{k=1}^\infty \left[ \frac{1}{p^2}q^{-k} b_k -
\frac{q^{-k}b_{k-1}}{c_{k}}\right] e_k$.

The fact that $v_{12}$ satisfies the first of the equations in \eqref{eq_set_of_equations_for_eta_infty_v12} follows essentially from the computations displayed above. As regards the second one, it suffices to note that our choice of concrete realisation makes $\gamma$ self-adjoint, so that 
\[ \gamma^* v_{32} = \frac{1}{p^2} e_0+
\sum_{k=1}^\infty \left[ \frac{1}{p^2} b_k -
\frac{b_{k-1}}{c_{k}}\right] e_k = (\frac1{p^2}- \alpha^*)v_{12}.\]

Consider then the matrix $\Delta(c)\in M_3(\C)$, where $c\in Z^2(U_Q^+)$ is the cup product of $1$-cocycles associated with the matrix $V$. It is diagonal (recall that our $Q$ has one-dimensional eigenspaces), and the formula \eqref{Dform}  implies that 
\begin{equation}  \label{Deltacform} \Delta(c)_{kk} = \sum_{p=1}^3 \left(\langle v_{pk}, v_{pk}\rangle - \frac{Q_k}{Q_p}\langle v_{kp}, v_{kp}\rangle \right), \;\;\; k=1,\ldots,3.  \end{equation}
Thus finally, taking into account the form of the matrix $V$, we obtain
\[ \Delta(c)= \textup{diag}\left[-\frac{1}{p^2} \|v_{12}\|^2, \|v_{12}\|^2 + \|v_{32}\|^2, - \frac{q^2}{p^2} \|v_{32}\|^2   \right].\]
\end{proof} 
 
The next result is completely analogous, but deals with the case of $p \in( q^{\frac{1}{2}},q )$. 
 
\begin{prop} \label{cocycleCase2}
Let $0<q<p<1$, $p > q^{\frac{1}{2}}$. Then the set of equations  \eqref{eq_set_of_equations_for_eta_infty_v2} admits a solution $(v_{21},v_{23})$ with both  $v_{12},v_{32}$ non-zero vectors in $\ell^2(\N_0)$. Put 
$$V=\begin{pmatrix}
0 & 0 & 0 \\ v_{21} & 0 & v_{23} \\ 0 & 0 & 0 
\end{pmatrix}$$
and define the corresponding cocycle $\eta \in Z_1(\AuQ, _\pi \ell^2(\N_0) _\e)$ via Lemma \ref{lem_QRV}
(recall that $Q={\rm diag}\ [1,p^2,q^2]$, and the representation $\pi$ is defined via the matrix \eqref{formula:Keyrep}). Let $c\in Z_2(\AuQ)$ be the cocycle defined as the cup product of 	$\eta$ with itself, as in Lemma \ref{cupproduct}. Then 
\[\Delta(c) \textup{ is an invertible diagonal matrix},\]
where $\Delta$ is the defect map of Definition \ref{def:defect}. 	
\end{prop}
\begin{proof}
As the proof follows the identical steps as that of the last proposition, we just present key differences.	

This time we first suppose that we have a solution with $v_{23}=\sum_{k=0}^\infty b_ke_k$, where $(b_k)_{k=0}^\infty \in \ell^2$. Then the first of the  equations in  \eqref{eq_set_of_equations_for_eta_infty_v12} would yield
\begin{align*}
0 & = 
\sum_{k=0}^\infty (1+\frac{p^4}{q})b_ke_k - 
\sum_{k=1}^\infty \frac{p^2}{q^2} \sqrt{1-q^{2k}} b_ke_{k-1}
-\sum_{k=0}^\infty  qp^2b_k\sqrt{1-q^{2(k+1)}}e_{k+1}.
\\ & =
\left((1+\frac{p^4}{q})b_0  -\frac{p^2}{q^2} \sqrt{1-q^{2}} b_1\right) e_0
\\& \;\;\;\;+
\sum_{k=1}^\infty \left((1+\frac{p^4}{q})b_k - \frac{p^2}{q^2} 
\sqrt{1-q^{2(k+1)}} b_{k+1}
-qp^2\sqrt{1-q^{2k}}b_{k-1}
\right) e_k .
\end{align*}
and we are led to the recurrence formula (with $k \in \N$)
\begin{equation}\label{eq_bks1b}
b_1 = \frac{q^2}{p^2}\frac{(1+\frac{p^4}{q})}{ \sqrt{1-q^{2}} }b_0 
% b_1 =\frac{q^2(\frac{1}{q}+s^2)}{s\sqrt{1-q^{2}}}b_0
, \quad 
b_{k+1}
=\frac{q^2}{p^2 \sqrt{1-q^{2(k+1)}}}[(1+\frac{p^4}{q})b_k - 
qp^2\sqrt{1-q^{2k}}b_{k-1}].
\end{equation}
Set then 
$\kappa:=\frac{q^2}{p^2}(1+\frac{p^4}{q})=\frac{q^2}{p^2}+qp^2$, 
$\mu:=q^3$ and $c_k$ ($k \in \N$) as before. Then 
\eqref{eq_bks1b} takes the form identical to \eqref{eq_bks2}.

Set $b_0=1$ and define $b_k$ for $k \in \N$ inductively by \eqref{eq_bks2}.
Lemma \ref{lem_rysiek}, with $a=qp^2 $ and $b=\frac{q^2}{p^2}$ (so that $b<q$), shows then that $(q^{-k} b_k)_{k=0}^\infty \in \ell^2$. Thus we can indeed consider the following (non-zero!) vectors in $\ell^2(\N_0)$: $v_{23}:=\sum_{k=0}^\infty b_ke_k$ and $v_{21}:=\sum_{k=0}^\infty \left[ \frac{1}{p^2}q^{-k} b_k -
\frac1{q}\frac{q^{-(k+1)}b_{k+1}}{c_{k+1}}\right] e_k$. Then we verify that these indeed satisfy \eqref{eq_set_of_equations_for_eta_infty_v12}. 

If now $c\in Z^2(U_Q^+)$ is the cup product of $1$-cocycles associated with the matrix $V$ as in the formulation of the theorem, formula \eqref{Deltacform} remains valid and we finally obtain
\[ \Delta(c)= \textup{diag}\left[\|v_{21}\|^2, - p^2\|v_{21}\|^2 - \frac{p^2}{q^2}\|v_{23}\|^2,  \|v_{23}\|^2   \right].\]	
\end{proof}

\begin{remark}
One might ask whether the cocycles corresponding to a matrix $V\in M_3(\ell^2(\mathbb{N}_0))$ of the form
\[
V=\left(
\begin{array}{ccc}
 v_{11} & 0 & v_{13} \\
 0 & 0 & 0 \\
 v_{31} & 0 & v_{33}
\end{array}
\right)
\]
could also be used to produce non-trivial 2-cocycles. For that the vectors $v_{11},v_{13},v_{31},v_{33}$ would have to satisfy the equations in \eqref{eq-set-suq2}, and the first of these implies that the coefficients of $v_{13}=\sum_{k=0}^\infty a_ke_k$ and  $v_{31}=\sum_{k=0}^\infty b_ke_k$ satisfy the equalities
\[
-q^{-1}a_k = b_k, \qquad k\in\mathbb{N}_0.
\]
Therefore we would have $\|v_{13}\|^2=q^2\|v_{31}\|^2$. This implies that
\begin{eqnarray*}
D_{11} &=& A_{11} - B_{11} = \|v_{31}\|^2 - \frac{1}{q^2} \|v_{13}\|^2 = 0, \\
D_{33} &=& A_{33} - B_{33} = \|v_{13}\|^2 - q^2 \|v_{31}\|^2 =0,
\end{eqnarray*}
so that we would obtain
\[
\Delta(c) = 0
\]
for the associated 2-cocycle $c$. In other words, the 2-cocycle produced from such a 1-cocycle using the cup product is always a coboundary.
\end{remark}

We are ready for the main result of this subsection; we shall use the above constructions in dimension $3$ to deal with the general case of a matrix $Q$ of arbitrary size with three distinct eigenvalues which do not form a geometric progression.

\begin{thm}\label{thm:coh3}
Suppose that $d \in \N$ and that $Q \in M_d(\C)$ is a strictly positive matrix with precisely three distinct eigenvalues. If these eigenvalues do not form a geometric progression, then
\[ H^2(U_Q^+) \simeq sl_Q(d).	\]
\end{thm}

\begin{proof}
Consider the list of eigenvalues of $Q$. By rescaling we can assume that it is of the form 	
$(1,p^2,q^2)$, with $0<q<p<1$; we also assume that $Q$ is diagonal, and the dimensions of respective eigenspaces equal $d_1$, $d_2$ and $d_3$. 

Let then $\tilde{Q}\in M_3(\C)$, $\tilde{Q}= {\rm diag}\ [1,p^2,q^2]$. It is easy to see that similarly as in Section \ref{sec:prelim} we have the inclusion $U^+_{\tilde{Q}} \subset U_Q^+$, realised via the surjective $^*$-homomorphism $q: \AuQ \to A_u(\tilde{Q})$, mapping generators of $\AuQ$ either to generators of $A_u(\tilde{Q})$ or to $0$ or to $1$  (the choice here depends on  fixing a basis vector in each of the three eigenspaces of $Q$, thus choosing three basis vectors in $\C^d$).
Let then $c$ denote a cocycle on $A_u(\tilde{Q})$ constructed via Proposition \ref{cocycleCase1} or Proposition \ref{cocycleCase2}, depending on whether $p>q^\frac{1}{2}$ or  $p<q^\frac{1}{2}$.  Put
$c'= c \circ (q \ot q): \AuQ \ot \AuQ\to \C$, so that $c' \in Z^2(U_Q^+)$. It is then easy to check, based on conclusions of Propositions \ref{cocycleCase1} or \ref{cocycleCase2} (and remembering that we are dealing with normalised cocycles) that 
$\Delta(c')$ is a diagonal matrix with one non-zero entry in each eigenspace of $Q$. Thus in particular $\Delta(c') \notin sl(d_1)\oplus sl(d_2) \oplus sl(d_3)$. 

Now Theorem \ref{thm:image}, Corollary \ref{cor:H2} and the dimension count yield the desired result.
\end{proof}

\subsection*{Higher `block' dimensions and the main result}
It turns out that the techniques of the last subsection allow us to deduce a complete result for matrices $Q$ with more than three distinct eigenvalues. We begin with a basic combinatorial lemma. Note that the number $n-2$ is optimal, if we want conditions (i)-(iii) below to hold.

\begin{lemma}\label{triples}
Let $n \in \N$, $n \geq 4$, and suppose that we are given a tuple of distinct positive numbers $(q_1, \ldots, q_n)$. Then one can choose $n-2$ three-element subsets of $\{1,\ldots,n\}$ such that the elements $q_{n_1}, q_{n_2}, q_{n_3}$ corresponding to a given subset do not form (in any order) a geometric progression. Moreover this can be done in such a way that 
\begin{itemize}
\item[(i)]
there exists an element in $\{1,\ldots,n\}$, say $j_n$, which belongs to precisely one of the chosen triples;
\item[(ii)] after we delete from our collection of triples the one corresponding to $j_n$, there exists an element in $\{1,\ldots,n\}\setminus \{j_n\}$ which belongs to precisely one of the remaining chosen triples;
\item[(iii)] this behaviour continues until we are left with just $2$ triples (containing elements from a $4$-element subset $A$ of  $\{1,\ldots,n\}$), and also then  there exist two  elements in $A$ each of which belongs to precisely one of the two remaining chosen triples.
\end{itemize}	
\end{lemma}
\begin{proof}
Begin with $n=4$. We can assume that the sequence	$(q_1, \ldots, q_4)$ is increasing. We need to find two `bad' triples. Suppose say that $(q_1,q_2, q_3)$ form a geometric progression. Then this is not the case for $(q_1,q_2,q_4)$ and either for $(q_1, q_3, q_4)$ or for $(q_2, q_3, q_4)$. By symmetry we can also find two `bad' triples if $(q_2,q_3, q_4)$  form a geometric progression. And if neither $(q_1,q_2, q_3)$ nor $(q_2,q_3, q_4)$ form a geometric progression, we have already located two `bad' triples. It is easy to see that in each case there are two elements in $\{1,2,3,4\}$ each of which  belongs to precisely one of the triples.
	
The general statement follows inductively; indeed, suppose that the result has been proved for a given $n\geq 4$. Consider then the tuple $(q_1, \ldots, q_{n+1})$, which without loss of generality can be ordered so that the numbers are increasing. We know that we can find  $n-2$ relevant triples of elements in $\{1,\ldots,n\}$. Then it is enough to  note that there must exist a pair $(k_1,k_2)$ of distinct elements in $\{1,\ldots,4\}$ such that the  elements $q_{k_1}, q_{k_2}, q_{n+1}$ do not form a geometric progression. This can be seen similarly as in the first part. Obviously $n+1$ belongs to only one of the triples in the new set; and once we remove it (together with the corresponding triple) from the list, we are back in the  case already assumed to satisfy conditions (i)-(iii), so that the conditions (i)-(iii) hold also for the new choice.
\end{proof}

Recall the formula \eqref{slQd} defining  $sl_Q(d)$. We are ready to state and prove the main result of the paper, stated already in the introduction. It will essentially follow from the arguments in the proof of Theorem \ref{thm:coh3} and the last lemma.

\begin{thm} \label{theoremmain}
Let $d \in \N$ and let $Q\in M_d(\C)$ be a strictly positive matrix, whose eigenvalue list is not of the form $(p, pq,pq^2)$ with $p>0, q\in(0, \infty) \setminus\{1\}$. 
Then \[ H^2(U_Q^+) \simeq sl_Q(d).\]			
\end{thm}
\begin{proof}
Suppose that the eigenvalue list of $Q$ is as follows: $(q_1, \ldots, q_n)$. When $n=1$ the result is \cite[Theorem 4.6]{DFKS18}; for $n=2$ it is Corollary \ref{cor:coh2}, and for $n=3$ (when the only exceptional case features) it is Theorem \ref{thm:coh3}.

Assume then that $n\geq4$.  We can suppose that $Q$ is diagonal, so that we have split $\{1, \ldots,d\}$ into $n$ parts corresponding to consecutive eigenspaces. For each of these parts denote by $\textup{Tr}_k$ the corresponding partial trace, as in the proof of Proposition \ref{prop:nogo}: formally
\[\textup{Tr}_k (A) = \sum_{r\in \{1,\ldots, d\,: \,Q_r = q_k\}} A_{rr}, \;\;\;\;\;\; A \in M_d(\C), \;k=1, \ldots,n.  \] 
Choose for $(Q_1, \ldots, Q_n)$ a collection  $\mathcal{J}$ of triples of elements of $\{1,\ldots,n\}$ as in Lemma \ref{triples}.
For each $\mathbf{j} \in \Jnd$, say $\mathbf{j}=(j_1, j_2, j_3)$ define the matrix $Q_\mathbf{j} = \textup{diag}[q_{j_1}, q_{j_2}, q_{j_3} ]$, and as in the proof of Theorem  \ref{thm:coh3} first use the natural inclusion $U_{Q_\mathbf{j}}^+ \subset U_Q^+$ and then the cocycle constructed on $A_u(Q_\mathbf{j})$ either in  Proposition \ref{cocycleCase1} or in Proposition \ref{cocycleCase2} (our choice of $\mathcal{J}$ guarantees that one of these applies) to produce a cocycle $c_\mathbf{j} \in Z^2(U_Q^+)$ which has the following property: 
\begin{equation}\label{keyproperty}\textup{Tr}_k\Delta(c_\mathbf{j}) \neq 0  \textup{ if and only if } k \in \mathbf{j};\end{equation}
note that by Theorem \ref{thm:image} each $\Delta(c_\mathbf{j})$ belongs to $sl_Q(d)$.

Let  $B=\{A_1, \ldots, A_l\}$ be a linear basis of $sl(d_1)\oplus \ldots \oplus sl(d_n)$; obviously we have $\textup{Tr}_k(A_p)=0$ for each $p=1, \ldots,l$, $k =1, \ldots,n$. We claim that the collection $B \cup \{\Delta(c)_\mathbf{j}: \mathbf{j} \in \Jnd\}$ forms a linear basis in $sl_Q(d)$. The dimension count implies that it suffices to show that these vectors are linearly independent. Consider then a matrix $X=\sum_{p=1}^l c_p A_p + \sum_{\mathbf{j}\in \Jnd} d_{\mathbf{j}} \Delta(c)_\mathbf{j}$, with $c_1,\ldots,c_k \in \C$, $d_\mathbf{j} \in \C$ for all $\mathbf{j} \in \Jnd$ and suppose that $X=0$. 
Let then $j_n$ be the element of $\{1, \ldots, n\}$ which belongs to precisely one triple in $\Jnd$ (as guaranteed by condition (i) in Lemma \ref{triples}), say $\mathbf{j}_n$.  Then by \eqref{keyproperty} and vanishing of the partial traces of all $A_p$ we have
\[ 0 = \textup{Tr}_{j_n} (X) = d_{\mathbf{j}_n} \textup{Tr}_{j_n}  \mathbf{j}_n.\]
As $\textup{Tr}_{j_n}  \mathbf{j}_n \neq 0$ by \eqref{keyproperty}, we must have $d_{\mathbf{j}_n} =0$. We can then ignore $j_n$, discard the corresponding triple $\mathbf{j}_n$ and use condition (ii) in Lemma \ref{triples} to find the next element $\mathbf{j}\in \Jnd$ for which $d_\mathbf{j}=0$. We continue this until we are left just with 2 triples, say $\mathbf{j}_1$ and $\mathbf{j_2}$ (and a four element subset of $\{1, \ldots,n\}$). Then we use condition (iii) in Lemma \ref{triples} to find two partial traces which will show that also coefficients   $d_{\mathbf{j}_1}, d_{\mathbf{j}_2}$ vanish. Thus $\sum_{p=1}^l c_p A_p=0$, and as $\{A_1, \ldots, A_l\}$ was a basis, in fact all coefficients $c_p$ also vanish. This ends the proof.
\end{proof}

\begin{remark}
	We believe that the conclusion of the above theorem should hold for all matrices $Q$.
The difficulty with the missing case lies in the fact that we cannot decide whether for a given $q\in (0,1)$ either \begin{equation*} %\label{eq_set_of_equations_for_eta_infty_v12}
	\begin{cases}
	\left[(1+ q)I-\frac1{q}\alpha - q^2\alpha^* \right]v_{12}=0 \\ 
	\gamma^*v_{32}= (\frac1{q}- \alpha^*)v_{12} 
	\end{cases}, 
	\end{equation*}
	or 
	\begin{equation*} %\label{eq_set_of_equations_for_eta_infty_v2}
	\begin{cases}
	[(1+q)I- \frac{1}{q}\alpha  -q^2 \alpha^*]v_{23} =0\\
	\gamma^*v_{21} =(\frac1{q} -\frac{1}{q^2}\alpha)v_{23}
	\end{cases} 
	\end{equation*}
	admits a non-zero solution (a pair of vectors in $\ell^2(\N_0)$). If the answer were positive, we would be able to argue as in the last subsection to first complete the case of $Q$ with three eigenspaces and then obtain the complete result via Lemma \ref{triples}.
\end{remark}

\begin{remark}
Careful analysis of the proof of Theorem \ref{theoremmain}, including the earlier constructions in this section and these in \cite[Subsection 4.1]{DFKS18}, permits identifying explicit $2$-cocycles which yield a linear basis in $H^2(U_Q^+)$ (excluding the exceptional case, not covered by Theorem \ref{theoremmain}).
\end{remark}

\section*{Appendix}
In the appendix we state and prove a lemma regarding square summability of certain sequences defined by recurrence relations, which was kindly provided to us by Ryszard Szwarc.

\begin{customthm}{A.1}\label{lem_rysiek}
Let $q\in (0,1)$, let $a, b>0$ and put $\kappa=a+b$ and $\mu=a\cdot b$. Let 
	$c_k:=(1-q^{2k})^{-\frac12}$, $k \in \N_0$. Let $(b_k)_{k\in \N_0}$ be the sequence defined by 
	the recurrence relation
	\begin{equation} \label{eq_bks3}
	b_0=1, \quad b_1=\kappa c_1, \quad 
	b_{k+1} = \kappa c_{k+1}b_k -\mu\frac{c_{k+1}}{c_k}b_{k-1}
	\quad \mbox{for}\quad k\geq 1.
	\end{equation}
	Then 
	$$\lim_{k\to \infty}{b_k\over b_{k-1}} = \max(a,b).$$
	Consequently, the sequence $(q^{-k}b_k)_{k\in \N_0}$ is square summable if and 
	only if $\max(a,b)<q$.  
\end{customthm}

\begin{proof}
Write the recurrence relation as
	$$b_k={1\over \kappa 
		c_{k+1}}b_{k+1}+{\mu \over \kappa c_k}b_{k-1}$$
and define
	\begin{equation} \label{eq_sequence_gk}
	g_0=0,\qquad g_k={\mu \over \kappa c_k}{b_{k-1}\over b_k}, \quad k\ge 1.  
	\end{equation}
Then
	\begin{align*}
	1&=\frac{1}{\kappa c_{k+1}}\frac{b_{k+1}}{b_k}+{\mu \over \kappa 
		c_k}\frac{b_{k-1}}{b_k}
	= \frac{1}{\kappa c_{k+1}}\frac{\mu}{\kappa c_{k+1}g_{k+1}}+g_k,
	\end{align*}
from which we deduce
	$$g_k(1-g_{k-1})=\frac{\mu}{\kappa^2c_k^2},\quad k\ge 1.$$
The sequence on the right hand side is increasing and convergent to 
	$\mu/\kappa^2.$
	Therefore the sequence $(g_k)_{k=0}^\infty$ is increasing. Indeed, $g_0=0\le 
	g_1=\mu/(\kappa 
	c_1^2).$
Fix $k \in \N$ and suppose we showed that  $g_{k-1}\le g_k.$ Then 
	$$ g_{k+1}(1-g_k)=\frac{\mu}{\kappa^2c_{k+1}^2} \ge \frac{\mu}{\kappa^2c_k^2} = 
	g_k(1-g_{k-1}) \ge g_k(1-g_k).$$	
Hence $g_{k+1}>g_k.$ 
	
We also have 
	\begin{equation} \label{eq_gk_bounded}
	g_k(1-g_{k-1})=\frac{\mu}{\kappa^2c_k^2}\le \frac14,\quad k\ge 1.
	\end{equation}
Indeed, $\frac{\mu}{\kappa^2c_k^2}\leq \frac{\mu}{\kappa^2}=\frac{ab}{(a+b)^2}$ 
and $4ab\leq (a+b)^2$. 
The estimate \eqref{eq_gk_bounded} yields $g_k\le \frac12.$ Indeed, $g_0\le 
	1/2.$ And if we know that $g_{k-1}\le 1/2$ for some $k\ge 1$,
then 
	$${1\over 2}g_k\le g_k(1-g_{k-1})\le {1\over 4},$$ i.e. $g_k\le 1/2.$	Thus the sequence $(g_k)_{k=1}^\infty$ is convergent to a number $g\le \frac12$ such that
	$$g(1-g)={\mu\over \kappa^2}.$$ 
	% The solutions are
	% $${1\over 2}\left ({1\pm \sqrt{1-4\mu/\kappa^2}}\right ).$$
Therefore
	$$g={1\over 2}\left ({1- \sqrt{1-4\mu/\kappa^2}}\right ).$$
In view of \eqref{eq_sequence_gk} this implies
	$$\lim_{k\to \infty}{b_k\over b_{k-1}}=\lim_{k\to \infty} {\mu \over \kappa c_k g_k} = \frac{\kappa 
		+\sqrt{\kappa^2-4\mu}}{2} .$$
	Using the assumption $\kappa=a+b$, $\mu=a\cdot b$, $a,b>0$ we get
	$$\lim_{k\to \infty}{b_k\over b_{k-1}} = \frac{(a+b) +|a-b|}{2} = \max(a,b).$$

The last conclusion follows from the d'Alembert's ratio test.

\begin{comment}	Consequently, if 
	$$\lim_k\frac{q^{-k}b_k}{q^{-(k-1)}b_{k-1}}= \frac1{q}\lim_k{b_k\over 
		b_{k-1}}= \frac1{q} \max(a,b)<1,$$
	then the sequence $(q^{-k}b_k)_k$ is square summable. 
	
	In order to determine the behaviour for $\max(a,b)\geq q$
	observe that
	$$ {\mu \over \kappa c_k}{b_{k-1}\over b_k}=g_k\le g.$$
	Thus
	$$\frac{q^{-k}b_k}{q^{-(k-1)}b_{k-1}}={\mu\over q\kappa g_k c_k}\ge {\mu\over 
		\kappa qg c_k}=
	\frac{\max(a,b)}{qc_k}.$$
	We get
	$$q^{-k}b_k \ge \left[\frac{\max(a,b)}{q}\right]^k 
	(c_1c_2\ldots c_k)^{-1}.$$ The product $c_1c_2\ldots c_k$ is increasing and  
	convergent to a positive number $c.$ Hence
	$$q^{-k}b_k\ge c \left [\max\left (qp^2,\, {q^2\over p^2}\right )\right ]^k.$$
	The latter allows to conclude the divergence of the series when $\max(a,b)\geq 
	q$.
	
\end{comment}
\end{proof}

\section*{Acknowledgements}
AK was supported by the Polish National Science Center grant  2016/21/D/ST1/03010. AS was partially supported by the National Science Center grant 2014/14/E/ST1/00525. We acknowledge support by the French 
MAEDI and MENESR and by the Polish National Agency for Academic Exchange (NAWA) in frame of the POLONIUM programme \\PPN/BIL/2018/1/00197/U/00021.
UF was supported by the French `Investissements d’Avenir' program, project ISITE-BFC
(contract ANR-15-IDEX-03), and by an ANR project (No./ANR-19-CE40-0002).
Finally we are very grateful to Ryszard Szwarc for his help regarding the lemma in the appendix.

\end{document}